\documentclass[11pt]{article}
\usepackage{amsthm}
\usepackage{latexsym,amssymb,amsmath}
\usepackage{graphicx,float,color,fancybox,shapepar,setspace,hyperref}
\usepackage{subfigure}
\usepackage{mathrsfs}
\usepackage{pgf,tikz}
\usetikzlibrary{arrows}
\usetikzlibrary{shapes.geometric}
\newtheorem{thm}{Theorem}[section]
\newtheorem{case}[]{Case}

\newtheorem{conj}[thm]{Conjecture}

\newtheorem{claim}[]{Claim}

 \def\dfn#1{{\sl #1}}
 \def\qed{\hfill\square}

\def\qed{ \hfill $\blacksquare$}

\begin{document} 
\title{Gallai-Ramsey numbers for graphs\\with chromatic number three}

\author{Qinghong Zhao\thanks{Corresponding Author. Supported by the Summer Graduate Research Assistantship Program of Graduate School. E-mail address: qzhao1@olemiss.edu (Q. Zhao).}~~and~Bing Wei\thanks{Supported in part by the summer faculty research grant from CLA at the University of Mississippi. E-mail address: bwei@olemiss.edu (B. Wei).}\\ 
\small Department of Mathematics, University of Mississippi, University, MS 38677,  USA}

\date{}
\maketitle
\begin{abstract}
Given a graph $H$ and an integer $k\ge1$, the Gallai-Ramsey number $GR_k(H)$ is defined to be the minimum integer $n$ such that every $k$-edge coloring of the complete graph $K_n$ contains either a rainbow (all different colored) triangle or a monochromatic copy of $H$. In this paper, we study Gallai-Ramsey numbers for graphs with chromatic number three such as $\widehat{K}_m$ for $m\ge2$, where $\widehat{K}_m$ is a kipas with $m+1$ vertices obtained from the join of $K_1$ and $P_m$, and a class of graphs with five vertices, denoted by $\mathscr{H}$. We first study the general lower bound of such graphs and propose a conjecture for the exact value of $GR_k(\widehat{K}_m)$. Then we give a unified proof to determine the Gallai-Ramsey numbers for many graphs in $\mathscr{H}$ and obtain the exact value of $GR_k(\widehat{K}_4)$ for $k\ge1$. Our outcomes not only indicate that the conjecture on $GR_k(\widehat{K}_m)$ is true for $m\le4$, but also imply several results on $GR_k(H)$ for some $H\in \mathscr{H}$ which are proved individually in different papers.\\
 
\end{abstract}

\section{Introduction}
In this paper, we only deal with finite, simple and undirected graphs. Given a graph $G$ and the vertex set $V(G)$, let $|G|$ denote the order of $G$ and $G[W]$ denote the subgraph of $G$ induced by a set $W\subseteq V(G)$. Given disjoint vertex sets $X,Y\subseteq V(G)$, if each vertex in $X$ is adjacent to all vertices in $Y$ and all the edges between $X$ and $Y$ are colored with the same color, then we say that $X$ is \dfn{$mc$-adjacent} to $Y$, that is, $X$ is \dfn{blue-adjacent} to $Y$ if all the edges between $X$ and $Y$ are colored with blue. We use $P_n$, $C_n$ and $K_n$ to denote the path, cycle and complete graph on $n$ vertices, respectively; and $S_n$ to denote a star on $n+1$ vertices. A kipas $\widehat{K}_m$ is obtained by deleting one edge on the rim of a wheel with $m+1$ vertices. For an integer $k\ge1$, we define $[k]=\{1,\ldots,k\}$. \medskip

The complete graphs under edge coloring without rainbow triangle usually have pretty interesting and somehow beautiful structures. In 1967, Gallai studied the structure under the guise of transitive orientations and obtained the following result \cite{Gallai} which was restated in \cite{Gy} in the terminology of graphs. 
\begin{thm}[\cite{Gallai},\cite{Gy}]\label{Gallai}
For a complete graph $G$ under any edge coloring without rainbow triangle, there exists a partition of $V(G)$ (called a Gallai-partition) with parts $V_1, V_2, \dots, V_\ell$, $\ell\ge2$, such that there are at most two colors on the edges between the parts and only one color on the edges between each pair of parts.
\end{thm}
In honor of Gallai's result, the edge coloring of a complete graph without rainbow triangle is called \dfn{Gallai coloring}. Given graphs $H_1,\ldots, H_k$ and an integer $k \ge 1$, the \dfn{Ramsey number} $R(H_1,\ldots, H_k)$ is defined to be the minimum integer $n$ such that every $k$-edge coloring of $K_n$ contains a monochromatic copy of $H_i$ in color $i\in [k]$, and the \dfn{Gallai-Ramsey number} $GR(H_1,\ldots,H_k)$ is defined to be the minimum integer $n$ such that every Gallai $k$-coloring of $K_n$ contains a monochromatic copy of $H_i$ in color $i\in [k]$. We simply write $R_k(H_1)$ and $GR_k(H_1)$ when $H_1=\dots=H_k$, and $R((k-s)H_{s+1},sH_1)$ and $GR_k((k-s)H_{s+1},sH_1)$ when $H_1=\dots=H_s$ and $H_{s+1}=\dots=H_k$. Clearly, $GR_2(F,H)=R(F,H)$ and $GR_k(H) \leq R_k(H)$ for all $k\ge1$. However, determining $GR_k(H)$ for a graph $H$ is far from trivial, even for a small graph. In 2010, the general behavior of $GR_k(H)$ for $H$ was established in \cite{exponential}. Given a Gallai-partition $V_1,\dots, V_\ell$ of a complete graph $G$, we define $\mathcal{G}=G[\{v_1,\dots,v_\ell\}]=K_{\ell}$ as a \dfn{reduced graph} of $G$, where $v_i\in V_i$ for all $i\in[\ell]$. Obviously, there exists a monochromatic copy of $H$ in $\mathcal{G}$ if $\ell\ge R_2(H)$, which leads to a monochromatic copy of $H$ in $G$. 
\begin{thm} [\cite{exponential}]
Let $H$ be a fixed graph  with no isolated vertices. Then $GR_k(H)$ is exponential in $k$ if $H$ is not bipartite, linear in $k$ if $H$ is bipartite but not a star, and constant (does not depend on $k$) when $H$ is a star.		
\end{thm}
In 2015, Fox, Grinshpun and Pach \cite{FGP} posed a conjecture for $GR_k(K_t)$. It is worth mentioning that the cases $t=3,4$ were proved in \cite{chgr},\cite{exponential} and \cite{k4}. Recently, Magnant and Schiermeyer \cite{colt} have made a breakthrough for the case $t=5$ while there are no any results for the cases $t\ge 6$. More information on this topic can be found in \cite{FGP, FMO}.
\begin{center} 
	
	\def\r{6pt}
	\def\dy{1cm}
	\tikzset{c/.style={draw,circle,fill=black,minimum size=\r,inner sep=0pt,
			anchor=center},
		d/.style={draw,circle,fill=black,minimum size=\r,inner sep=0pt, anchor=center}}
	\scalebox{0.5}{
		\begin{tikzpicture}
		\pgfmathtruncatemacro{\Ncorners}{5}
		\node[ regular polygon,regular polygon sides=\Ncorners,minimum size=3cm] 
		(poly\Ncorners) {};
		\foreach\x in {1,...,\Ncorners}{
			\node[d] (poly\Ncorners-\x) at (poly\Ncorners.corner \x){};
		}

		\foreach\X in {1,...,\Ncorners}{
			\ifnum\X=2
			\foreach\Y in {1,3,5}{
				\draw (poly\Ncorners-\X) -- (poly\Ncorners-\Y);}
			\fi

			\ifnum\X=5
			\foreach\Y in {1}{
				\draw (poly\Ncorners-\X) -- (poly\Ncorners-\Y);}
			\fi
			\ifnum\X=3
			\foreach\Y in {4}{
				\draw (poly\Ncorners-\X) -- (poly\Ncorners-\Y);}
			\fi

		}
		
		\end{tikzpicture}
	}	
	\hfil
	\scalebox{0.5}{
		\begin{tikzpicture}
		\pgfmathtruncatemacro{\Ncorners}{5}
		\node[ regular polygon,regular polygon sides=\Ncorners,minimum size=3cm] 
		(poly\Ncorners) {};
		\foreach\x in {1,...,\Ncorners}{
			\node[d] (poly\Ncorners-\x) at (poly\Ncorners.corner \x){};
		}

		\foreach\X in {1,...,\Ncorners}{
			\ifnum\X=2
			\foreach\Y in {1,3,4,5}{
				\draw (poly\Ncorners-\X) -- (poly\Ncorners-\Y);}
			\fi

			\ifnum\X=5
			\foreach\Y in {1}{
				\draw (poly\Ncorners-\X) -- (poly\Ncorners-\Y);}
			\fi		
			
		}
		
		\end{tikzpicture}
	}		
	\hfil	
	\scalebox{0.5}{
		\begin{tikzpicture}
		\pgfmathtruncatemacro{\Ncorners}{5}
		\node[ regular polygon,regular polygon sides=\Ncorners,minimum size=3cm] 
		(poly\Ncorners) {};
		\foreach\x in {1,...,\Ncorners}{
			\node[d] (poly\Ncorners-\x) at (poly\Ncorners.corner \x){};
		}

		\foreach\X in {1,...,\Ncorners}{
			\ifnum\X=2
			\foreach\Y in {1,3,5}{
				\draw (poly\Ncorners-\X) -- (poly\Ncorners-\Y);}
			\fi

			\ifnum\X=5
			\foreach\Y in {1,4}{
				\draw (poly\Ncorners-\X) -- (poly\Ncorners-\Y);}
			\fi

		}
		
		\end{tikzpicture}
	}
	\hfil	
	\scalebox{0.5}{
		\begin{tikzpicture}
		\pgfmathtruncatemacro{\Ncorners}{5}
		\node[ regular polygon,regular polygon sides=\Ncorners,minimum size=3cm] 
		(poly\Ncorners) {};
		\foreach\x in {1,...,\Ncorners}{
			\node[d] (poly\Ncorners-\x) at (poly\Ncorners.corner \x){};
		}

		\foreach\X in {1,...,\Ncorners}{
			\ifnum\X=2
			\foreach\Y in {1,3,5}{
				\draw (poly\Ncorners-\X) -- (poly\Ncorners-\Y);}
			\fi

			\ifnum\X=5
			\foreach\Y in {1,4}{
				\draw (poly\Ncorners-\X) -- (poly\Ncorners-\Y);}
			\fi
			
			\ifnum\X=3
			\foreach\Y in {4}{
				\draw (poly\Ncorners-\X) -- (poly\Ncorners-\Y);}
			\fi	
			
		}
		
		\end{tikzpicture}
	}
	\hfil
	\scalebox{0.5}{
		\begin{tikzpicture}
		\pgfmathtruncatemacro{\Ncorners}{5}
		\node[ regular polygon,regular polygon sides=\Ncorners,minimum size=3cm] 
		(poly\Ncorners) {};
		\foreach\x in {1,...,\Ncorners}{
			\node[d] (poly\Ncorners-\x) at (poly\Ncorners.corner \x){};
		}
		
		\foreach\X in {1,...,\Ncorners}{
			\ifnum\X=2
			\foreach\Y in {1,3,5}{
				\draw (poly\Ncorners-\X) -- (poly\Ncorners-\Y);}
			\fi

			\ifnum\X=5
			\foreach\Y in {3,4}{
				\draw (poly\Ncorners-\X) -- (poly\Ncorners-\Y);}
			\fi
			\ifnum\X=3
			\foreach\Y in {4}{
				\draw (poly\Ncorners-\X) -- (poly\Ncorners-\Y);}
			\fi

		}
		
		\end{tikzpicture}
	}
	\hfil
	\scalebox{0.5}{
		\begin{tikzpicture}
		\pgfmathtruncatemacro{\Ncorners}{5}
		\node[ regular polygon,regular polygon sides=\Ncorners,minimum size=3cm] 
		(poly\Ncorners) {};
		\foreach\x in {1,...,\Ncorners}{
			\node[d] (poly\Ncorners-\x) at (poly\Ncorners.corner \x){};
		}
		
		\foreach\X in {1,...,\Ncorners}{
			\ifnum\X=2
			\foreach\Y in {3,5}{
				\draw (poly\Ncorners-\X) -- (poly\Ncorners-\Y);}
			\fi

			\ifnum\X=5
			\foreach\Y in {1,3,4}{
				\draw (poly\Ncorners-\X) -- (poly\Ncorners-\Y);}
			\fi
			\ifnum\X=3
			\foreach\Y in {4}{
				\draw (poly\Ncorners-\X) -- (poly\Ncorners-\Y);}
			\fi
		}
		
		\end{tikzpicture}
	}
	\begin{tikzpicture}	
	\node at (-5.8,0) {$H_1$};
	\node at (-3.4,0) {$H_2$};
	\node at (-1,0) {$H_3$};
	\node at (1.3,0) {$H_4$};
	\node at (3.6,0) {$H_5$};
	\node at (6,0) {$H_6$};
	\end{tikzpicture}
\end{center}
\begin{center} 
	
	\def\r{6pt}
	\def\dy{1cm}
	\tikzset{c/.style={draw,circle,fill=black,minimum size=\r,inner sep=0pt,
			anchor=center},
		d/.style={draw,circle,fill=black,minimum size=\r,inner sep=0pt, anchor=center}}
	\scalebox{0.5}{
		\begin{tikzpicture}
		\pgfmathtruncatemacro{\Ncorners}{5}
		\node[ regular polygon,regular polygon sides=\Ncorners,minimum size=3cm] 
		(poly\Ncorners) {};
		\foreach\x in {1,...,\Ncorners}{
			\node[d] (poly\Ncorners-\x) at (poly\Ncorners.corner \x){};
		}
		
		\foreach\X in {1,...,\Ncorners}{
			\ifnum\X=2
			\foreach\Y in {1,3,4,5}{
				\draw (poly\Ncorners-\X) -- (poly\Ncorners-\Y);}
			\fi

			\ifnum\X=5
			\foreach\Y in {1,3,4}{
				\draw (poly\Ncorners-\X) -- (poly\Ncorners-\Y);}
			
			\fi

		}
		
		\end{tikzpicture}
	}
	\hfil
	\scalebox{0.5}{
		\begin{tikzpicture}
		\pgfmathtruncatemacro{\Ncorners}{5}
		\node[ regular polygon,regular polygon sides=\Ncorners,minimum size=3cm] 
		(poly\Ncorners) {};
		\foreach\x in {1,...,\Ncorners}{
			\node[d] (poly\Ncorners-\x) at (poly\Ncorners.corner \x){};
		}
		
		\foreach\X in {1,...,\Ncorners}{
			\ifnum\X=1
			\foreach\Y in {2,3,4,5}{
				\draw (poly\Ncorners-\X) -- (poly\Ncorners-\Y);}
			\fi

			\ifnum\X=2
			\foreach\Y in {3,5}{
				\draw (poly\Ncorners-\X) -- (poly\Ncorners-\Y);}
			\fi
			\ifnum\X=5
			\foreach\Y in {4}{
				\draw (poly\Ncorners-\X) -- (poly\Ncorners-\Y);}
			\fi
			\ifnum\X=3
			\foreach\Y in {4}{
				\draw (poly\Ncorners-\X) -- (poly\Ncorners-\Y);}
			\fi

		}
		
		\end{tikzpicture}
	}
	\hfil
	\scalebox{0.5}{
		\begin{tikzpicture}
		\pgfmathtruncatemacro{\Ncorners}{5}
		\node[ regular polygon,regular polygon sides=\Ncorners,minimum size=3cm] 
		(poly\Ncorners) {};
		\foreach\x in {1,...,\Ncorners}{
			\node[d] (poly\Ncorners-\x) at (poly\Ncorners.corner \x){};
		}

		\foreach\X in {1,...,\Ncorners}{
			\ifnum\X=1
			\foreach\Y in {2,3,4,5}{
				\draw (poly\Ncorners-\X) -- (poly\Ncorners-\Y);}
			\fi

			\ifnum\X=2
			\foreach\Y in {3}{
				\draw (poly\Ncorners-\X) -- (poly\Ncorners-\Y);}
			\fi

			\ifnum\X=4
			\foreach\Y in {5}{
				\draw (poly\Ncorners-\X) -- (poly\Ncorners-\Y);}
			\fi
		}
		
		\end{tikzpicture}
	}
	\hfil
	\scalebox{0.5}{
		\begin{tikzpicture}
		\pgfmathtruncatemacro{\Ncorners}{5}
		\node[ regular polygon,regular polygon sides=\Ncorners,minimum size=3cm] 
		(poly\Ncorners) {};
		\foreach\x in {1,...,\Ncorners}{
			\node[d] (poly\Ncorners-\x) at (poly\Ncorners.corner \x){};
		}

		\foreach\X in {1,...,\Ncorners}{
			\ifnum\X=1
			\foreach\Y in {2,5}{
				\draw (poly\Ncorners-\X) -- (poly\Ncorners-\Y);}
			\fi

			\ifnum\X=2
			\foreach\Y in {5}{
				\draw (poly\Ncorners-\X) -- (poly\Ncorners-\Y);}
			\fi

			\ifnum\X=3
			\foreach\Y in {4}{
				\draw (poly\Ncorners-\X) -- (poly\Ncorners-\Y);}
			\fi
		}
		
		\end{tikzpicture}
	}
	\hfil			
	\scalebox{0.5}{
		\begin{tikzpicture}
		\pgfmathtruncatemacro{\Ncorners}{5}
		\node[ regular polygon,regular polygon sides=\Ncorners,minimum size=3cm] 
		(poly\Ncorners) {};
		\foreach\x in {1,...,\Ncorners}{
			\node[d] (poly\Ncorners-\x) at (poly\Ncorners.corner \x){};
		}

		\foreach\X in {1,...,\Ncorners}{
			\ifnum\X=2
			\foreach\Y in {1,3,4}{
				\draw (poly\Ncorners-\X) -- (poly\Ncorners-\Y);}
			\fi

			\ifnum\X=3
			\foreach\Y in {4}{
				\draw (poly\Ncorners-\X) -- (poly\Ncorners-\Y);}
			\fi

			\ifnum\X=5
			\foreach\Y in {1,3,4}{
				\draw (poly\Ncorners-\X) -- (poly\Ncorners-\Y);}
			\fi
		}
		
		\end{tikzpicture}
	}		
	\hfil
	\scalebox{0.5}{
		\begin{tikzpicture}
		\pgfmathtruncatemacro{\Ncorners}{5}
		\node[draw, regular polygon,regular polygon sides=\Ncorners,minimum size=3cm] 
		(poly\Ncorners) {};
		\foreach\x in {1,...,\Ncorners}{
			\node[d] (poly\Ncorners-\x) at (poly\Ncorners.corner \x){};
		}

		\foreach\X in {1,...,\Ncorners}{
			\ifnum\X=1
			\foreach\Y in {3,4}{
				\draw (poly\Ncorners-\X) -- (poly\Ncorners-\Y);}
			\fi		
			
		}
		
		\end{tikzpicture}
	}
	\begin{tikzpicture}	
	\node at (-6.8,0) {$H_7$};
	\node at (-4.5,0) {$H_8$};
	\node at (-2,0) {$H_9$};
	\node at (0.3,0) {$H_{10}$};
	\node at (2.6,0) {$H_{11}$};
	\node at (5,0) {$H_{12}$};
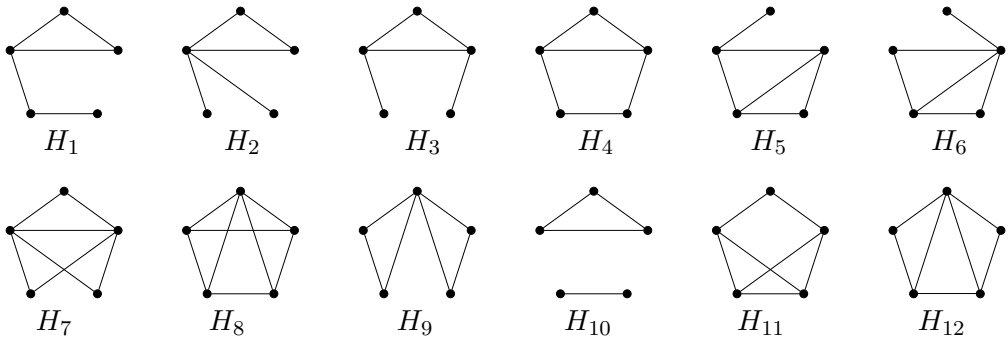
	\node at (-0.8,-0.8){Fig. 1. All isolated-free graphs with five vertices of chromatic number three.};
	\end{tikzpicture}
\end{center}

Let $\mathscr{H}$ denote a class of isolated-free graphs with five vertices of chromatic number three. It is known that $\mathscr{H}$ contains twelve graphs (see Fig. 1). Recently, the Gallai-Ramsey numbers for many graphs above have been determined in \cite{17} for $H_1$ and $H_2$, \cite{five} for the graphs from $H_3$ to $H_6$, \cite{JZ} for $H_7$, \cite{W2N,YM1} for $H_8$ and \cite{YM} for $H_9$. However, the Gallai-Ramsey numbers for the remaining graphs $H_{10}$, $H_{11}$ and $H_{12}$ are still unknown. Let $\mathscr{F}=\{H_1,H_2,H_3,H_4,H_5,H_6,H_{10},H_{11}\}$. In this paper, we first study the general lower bound of $GR_k(H)$ for all $H\in\mathscr{F}\cup\widehat{K}_m$ with $k\ge1$ and $m\ge2$ in Section 2. Then we give a unified proof to determine the exact values of $GR_k(H)$ for all $H\in \mathscr{F}\cup\{H_{12}\}$ in Section 3 under the enlightenment from \cite{W2N}. Therefore, the Gallai-Ramsey numbers for all graphs in $\mathscr{H}$ are determined. The following theorem is one of our main results.
\begin{thm}\label{thm} 
For all $k\geq1$, if $H\in\mathscr{F}-\{H_{10}\}$, then 
\[GR_k(H)=
\begin{cases}
(R_2(H)-1)\cdot5^{(k-2)/2}+1, & \textup{if $k$ is even,}\\
4\cdot5^{(k-1)/2}+1, & \textup{if $k$ is odd,}
\end{cases}\]
and if $H=H_{10}$, then
$GR_1(H)=5$, $GR_2(H)=7$ and 
\[GR_k(H)=
\begin{cases}
5^{k/2}+1, & \textup{if $k\ge 4$ is even,}\\
2\cdot5^{(k-1)/2}+1, & \textup{if $k\ge3$ is odd.}
\end{cases}\]
\end{thm}
Note that $K_3\cong\widehat{K}_2$, $K_4-e\cong\widehat{K}_3$ and $H_{12}\cong\widehat{K}_4$. In this paper, we come up with a new approach to construct the general lower bound of $GR_k(H)$ for some $H$, which will provide a direction for us to study the Gallai-Ramsey numbers of $\widehat{K}_m$ with $m\ge2$. So far, the exact values of $GR_k(\widehat{K}_2)$ and $GR_k(\widehat{K}_3)$ have been determined as follows:
\begin{thm}[\cite{chgr},\cite{exponential}] For all $k\ge1$,
\[GR_k(\widehat{K}_2)=\begin{cases}
	5^{k/2} + 1 & \textup{if $k$ is even,}\\
	2\cdot5^{(k-1)/2} + 1 & \textup{if $k$ is odd.}
	\end{cases}
	\]
\end{thm}
\begin{thm}[\cite{JZ}]\label{th} 
	For all $k\geq1$, $GR_1(\widehat{K}_3)=4$ and
	\[GR_k(\widehat{K}_3)=
	\begin{cases}
    9\cdot5^{(k-2)/2}+1, & \textup{if $k\ge2$ is even,}\\
   18\cdot5^{(k-3)/2}+1, & \textup{if $k\ge3$ is odd.}
	\end{cases}\]
\end{thm} 
In this paper, we prove the following result which implies the exact value of $GR_k(\widehat{K}_4)$. 
\begin{thm}\label{thmmm}
	Let $k\ge 1$ and $s$ be an integer with $0\le s\le k$. Then
	\[GR_k((k-s)P_3,s\widehat{K}_4)=
	\begin{cases}
	2\cdot5^{s/2}+1, & \textup{if $s$ is even and $s<k$,}\\
	2\cdot5^{s/2}, & \textup{if $s$ is even and $s=k$,}\\
	4\cdot5^{(s-1)/2}+1, & \textup{if $s$ is odd.}
	\end{cases}\]
\end{thm} 
One can obtain the Gallai-Ramsey number of $\widehat{K}_4$ from Theorem \ref{thmmm} by setting $s=k$. It is obviously true that $GR_1(\widehat{K}_m)=m+1$. With the support of the general lower bound of $GR_k(\widehat{K}_m)$ given in Section 2 and the exact values of $GR_k(\widehat{K}_m)$ for $2\le m\le 4$, we propose a conjecture as follows:
\begin{conj}
For all $k\ge2$ and $m\ge2$,
	\[GR_k(\widehat{K}_m)=
	\begin{cases}
	(R_2(\widehat{K}_m)-1)\cdot5^{(k-2)/2}+1, & \textup{if $k$ is even and $m$ is odd,}\\
	R_2(\widehat{K}_m)+(m/2)\cdot(5^{k/2}-5), & \textup{if $k$ is even and $m$ is even,}\\
	max\{2(R_2(\widehat{K}_m)-1),5m\}\cdot5^{(k-3)/2}+1, & \textup{if $k\ge3$ is odd.}
	\end{cases}\]	
\end{conj}
We shall make use of the following theorems in the proof of our main results.
\begin{thm}[\cite{4}]\label{FD}
	$GR_k(P_3)=3$ for all $k\ge1$.
\end{thm}
\begin{thm}[\cite{R2}]\label{3}
	$R_2(H_{10})=7$, $R_2(H^{'})=9$ and $R_2(H^{''})=10$, where $H^{'}\in\{H_1,H_2,H_3,H_4\}$ and $H^{''}\in\{H_5,H_6,H_{11},H_{12}\}$. 
\end{thm}
\begin{thm}[\cite{Kapps}]\label{5}
	$R(P_3,\widehat{K}_4)=5$.
\end{thm}

\section{Construction of a general lower bound}
In this section, we first give a general lower bound of $GR_k(H)$ for $H\in\mathscr{F}\cup\widehat{K}_m$ with $k\ge1$ and $m\ge2$ by construction. Suppose $H=H_{10}$ and $k\ge1$. Let $g(1)=|H|-1=4$, $g(2)=R_2(H)-1=6$ as $R_2(H)=7$ by Theorem \ref{3}, and let 
\[g(k)=\begin{cases}
5^{k/2}, & \textup{if $k\ge 4$ is even,}\\
2\cdot5^{(k-1)/2}, & \textup{if $k\ge3$ is odd.}
\end{cases}\]
Since $GR_1(H_{10})=|H_{10}|$, $GR_2(H_{10})=R_2(H_{10})$ and $K_3$ is a subgraph of $H_{10}$, we have $GR_k(H_{10})\ge g(k)+1$ for all $k\ge1$. Suppose $H\in\mathscr{F}-\{H_{10}\}$ and $k\ge1$. Let
\[
g(k)=
\begin{cases}
(R_2(H)-1)\cdot 5^{(k-2)/2},& \text{ if }  k  \text{ is even,}\\
4\cdot 5^{(k-1)/2},& \text{ if } k  \text{ is odd.}
\end{cases}
\]
Suppose $H=\widehat{K}_m$ for $m\ge2$ and $k\ge1$. Let $g(1)=|H|-1=m$ and	
\[
g(k)=
\begin{cases}
(R_2(H)-1)\cdot5^{(k-2)/2}, & \textup{if $k$ is even and $m$ is odd,}\\
R_2(H)+(m/2)\cdot(5^{k/2}-5)-1, & \textup{if $k$ is even and $m$ is even,}\\
max\{2(R_2(H)-1),5m\}\cdot5^{(k-3)/2}, & \textup{if $k\ge3$ is odd.}
\end{cases}
\] 	

In this paper, we use $(G,c)$ to denote a complete graph $G$ under the Gallai coloring $c$. For all $k\ge1$ and $m\ge2$, we now construct the complete graph $(G_k,c_k)$ on $g(k)$ vertices recursively which contains no monochromatic copy of $H\in \widehat{K}_m\cup(\mathscr{F}-\{H_{10}\}$), where $c_k: E(G_k)\xrightarrow{} [k]$ is a Gallai $k$-coloring. Let $K_5$ be colored with colors $k-1$ and $k$ without monochromatic triangle and $(G_{k-2},c_{k-2})$ be the construction on $g(k-2)$ vertices with colors in $[k-2]$. We consider the following three cases. One can obtain a general lower bound of $GR_k(H)$ for $H\in\mathscr{F}-\{H_{10}\}$ from the construction in Cases 1 and 2, and $H=\widehat{K}_m$ from the construction in Cases 1-3.
\begin{case}\label{C}
$k$ is even and $m$ is odd for all $H\in \widehat{K}_m\cup(\mathscr{F}-\{H_{10}\})$.
\end{case}
For $k=2$, let $G_2$ be 2-colored complete graph on $g(2)=R_2(H)-1$ vertices containing no monochromatic copy of $H$ under $c_2: E(G_2)\xrightarrow{} \{1,2\}$. For all even $k\ge4$, let $(G_k,c_k)$ be the construction obtained by replacing each vertex of $K_5$ with a copy of $(G_{k-2} ,c_{k-2})$ such that the colors on all edges between the corresponding copies of $(G_{k-2},c_{k-2})$ are consistent with $K_5$.
\begin{case}\label{CA}
$k$ is odd for all $H\in \widehat{K}_m\cup(\mathscr{F}-\{H_{10}\})$.
\end{case}
For $k=1$, let $G_1$ be 1-colored complete graph on $g(1)=|H|-1$ vertices containing no monochromatic copy of $H$ under $c_1: E(G_1)\xrightarrow{} \{1\}$. For all odd $k\ge3$, let $(G_{k-1},c_{k-1})$ be the construction obtained from Case \ref{C}. Note that $(G_{k-1},c_{k-1})$ is also obtained by same method as in Case \ref{C} for even $m$. If $2(R_2(H)-1)\ge5(|H|-1)$, then let $(G_{k},c_{k})$ be the construction obtained by joining two copies of $(G_{k-1},c_{k-1})$ such that all edges between two copies are colored with color $k$. If $2(R_2(H)-1)<5(|H|-1)$, then let $(G_k,c_k)$ be the construction obtained from the copies of $(G_{k-2},c_{k-2})$ by the same method as in Case \ref{C} under $K_5$.  
\begin{case}\label{CASE}
	$k$ is even and $m$ is even for $H=\widehat{K}_m$.
\end{case}
For $k=2$, let $G_2$ be 2-colored complete graph on $g(2)=R_2(\widehat{K}_m)-1$ vertices containing no monochromatic copy of $\widehat{K}_m$ under $c_2: E(G_2)\xrightarrow{} \{1,2\}$. For all even $k\ge4$, let $(G_k,c_k)$ be the construction obtained from the copies of $(G_{k-2},c_{k-2})$ and the following copies of $G^{'}_{k-1}$ and $G^{''}_{k-1}$ by the same method as in Case \ref{C} under $K_5$.\medskip

We first construct the Gallai $(k-1)$-colored complete graph $G^{'}_{k-1}$ ($resp$. $G^{''}_{k-1}$) for even $k\ge4$ with colors in $[k-1]$ ($resp$. $[k]\setminus\{k-1\}$) recursively which contains neither monochromatic copy of $\widehat{K}_m$ in color $i\in [k-2]$ nor $S_{m/2}$ in color $k-1$ ($resp$. $k$). Let $G^{'}$ ($resp$. $G^{''}$) be a complete graph on $m/2$ vertices with the only color $k-1$ ($resp$. $k$). Clearly, $G^{'}$ ($resp$. $G^{''}$) contains no $S_{m/2}$ in color $k-1$ ($resp$. $k$). For all even $k\ge4$, let $G^{'}_{k-3}$ ($resp$. $G^{''}_{k-3}$) be the construction with colors in $[k-4]\cup\{k-1\}$ ($resp$. $[k-4]\cup\{k\}$) such that $G^{'}_1=G^{'}$ ($resp$. $G^{''}_1=G^{''}$) if $k=4$. Let $K^{1}_5$ (a complete graph of order five) be colored with colors $k-3$ and $k-2$ without monochromatic triangle. Let $G^{'}_{k-1}$ ($resp$. $G^{''}_{k-1}$) be obtained from the copies of $G^{'}_{k-3}$ ($resp$. $G^{''}_{k-3}$) by the same method as in Case \ref{C} under $K^{1}_5$. Since $G^{'}_{k-3}$ ($resp$. $G^{''}_{k-3}$) is a Gallai $(k-3)$-colored complete graph which contains neither monochromatic copy of $\widehat{K}_m$ in color $i\in[k-4]$ nor $S_{m/2}$ in color $k-1$ ($resp$. $k$), $G^{'}_{k-1}$ ($resp$. $G^{''}_{k-1}$) is a desired construction on $(m/2)\cdot 5^{(k-2)/2}$ vertices without rainbow triangle.\medskip 

 Let $X_i,X_j$ be any two corresponding copies such that the color on the edges between $X_i$ and $X_j$ is color $i\in\{k-1,k\}$. Since there can not be a monochromatic copy of $\widehat{K}_m$ in $(G_k,c_k)$, the following properties must hold.
\begin{itemize}
	\item There is neither $S_{m/2}$ nor $P_m$ in color $i$ within either $X_i$ or $X_j$.
	\item If $X_i$ contains $S_{m/2-1}$ or $P_{m-1}$ in color $i$, then $X_j$ has no edge in color $i$.
\end{itemize}
By the properties above, there is at least one copy of $(G_{k-2},c_{k-2})$, at most two copies of $G^{'}_{k-1}$ and at most two copies of $G^{''}_{k-1}$ embedding in $K_5$. So, for all even $k\ge4$, $(G_k,c_k)$ can be obtained by a new way as follows:  replacing one vertex of $K_5$ with one copy of $(G_{k-2},c_{k-2})$, two vertices of $K_5$ with two copies of $G^{'}_{k-1}$ and two vertices of $K_5$ with two copies of $G^{''}_{k-1}$ (see Fig. 2. Color $k$ is on the edges of outer five cycle and color $k-1$ is on the edges of inner five cycle).
\begin{center}
	\def\r{4pt}
	\def\dy{1cm}
	\tikzset{c/.style={draw,circle,fill=black,minimum size=.2cm,inner sep=0pt,anchor=center},
		d/.style={draw,circle,fill=white,minimum size=1cm,inner sep=0pt, anchor=center}}
	\begin{tikzpicture}[font=\large]
	\pgfmathtruncatemacro{\Ncorners}{5}
	\node[draw, regular polygon,regular polygon sides=\Ncorners,minimum size=3cm] 
	(poly\Ncorners) {};
	\node[regular polygon,regular polygon sides=\Ncorners,minimum size=3.5cm] 
	(outerpoly\Ncorners) {};
	\foreach\x in {1}{
		\node[d] (poly\Ncorners-\x) at (poly\Ncorners.corner \x){$G^{''}_{k-1}$};
	}	
	\foreach\x in {4}{
		\node[d] (poly\Ncorners-\x) at (poly\Ncorners.corner \x){$G^{''}_{k-1}$};
	}

	\foreach\x in {2,3}{
		\node[d] (poly\Ncorners-\x) at (poly\Ncorners.corner \x){$G^{'}_{k-1}$};
		\node[d] (poly\Ncorners-\Ncorners) at (poly\Ncorners.corner \Ncorners){$G_{k-2}$};
	}
	
	\foreach\X in {1,...,\Ncorners}
	{\foreach\Y in {1,...,\Ncorners}{
			\pgfmathtruncatemacro{\Z}{abs(mod(abs(\Ncorners+\X-\Y),\Ncorners)-2)}
			\ifnum\Z=0
			\draw [line width=2mm, green](poly\Ncorners-\X) -- (poly\Ncorners-\Y);
			\fi
		}
	}
	
	\foreach\X in {1,...,\Ncorners}
	{\foreach\Y in {1,...,\Ncorners}{
			\pgfmathtruncatemacro{\G}{abs(mod(abs(\Ncorners+\X-\Y),\Ncorners)-4)}
			\ifnum\G=0
			\draw [line width=2mm, blue](poly\Ncorners-\X) -- (poly\Ncorners-\Y);
			\fi
		}
	}
	\node at (0,-2) {$G_k$};
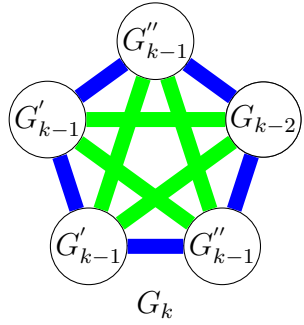
	\node at (0,-2.7) {Fig. 2. An example of the construction for all even $k\ge4$ and even $m\ge2$.};
	\end{tikzpicture}
\end{center}
Then $|G_k|=|G_{k-2}|+2|G^{'}_{k-1}|+2|G^{''}_{k-1}|$ for all even $k\ge4$ and even $m\ge2$. Thus we get the following recurrence equation:
\[\begin{cases}
g(k)=g(k-2)+2m\cdot5^{(k-2)/2}, & \text{$k\ge4$ is even,}\\
g(2)=R_2(\widehat{K}_m)-1.
\end{cases}\]
Note that $g(k)=g(k-2)+2m\cdot5^{(k-2)/2}$ is a nonhomogeneous equation. Since $5$ is not a solution of the characteristic equation $r^2-1=0$ corresponding to the homogeneous equation $g(k)-g(k-2)=0$, the particular solution of the nonhomogeneous equation is $(2m\cdot5^{(k-2)/2})\cdot a$. Then 
\[(2m\cdot5^{(k-2)/2})\cdot a=(2m\cdot5^{(k-4)/2})\cdot a+2m\cdot5^{(k-2)/2}\]
and thus $a=5/4$. Furthermore, the general solution of the homogeneous equation is

\[g^{'}(k)=b_1\cdot(-1)^{k}+b_2\cdot(1)^{k}.\] 
So the general solution of nonhomogeneous equation is 
\[g(k)=g^{'}(k)+(2m\cdot5^{(k-2)/2})\cdot(5/4)=b_1+b_2+(m/2)\cdot5^{k/2}\]
as $k$ is even. Since $g(2)=R_2(\widehat{K}_m)-1$, $b_1+b_2=R_2(\widehat{K}_m)-1-(5m/2)$. It follows that \[g(k)=R_2(\widehat{K}_m)+(m/2)\cdot(5^{k/2}-5)-1\]
for all even $k\ge2$ and even $m\ge2$.\medskip

Since the complete graphs $G_{k-1}$ and $G_{k-2}$ have no monochromatic copy of $H$ under $c_{k-1}: E(G_{k-1})\xrightarrow{} [k-1]$ and $c_{k-2}: E(G_{k-2})\xrightarrow{} [k-2]$ respectively in the corresponding cases, $(G_k,c_k)$ is a desired construction on $g(k)$ vertices without rainbow triangle. Therefore, $GR_k(H)\ge g(k)+1$ for all $H\in\widehat{K}_m\cup(\mathscr{F}-H_{10})$ with $k\ge1$ and $m\ge2$.\medskip

For $H=H_{10}$, we have
\[g(k-1)=
\begin{cases}
5^{(k-1)/2}, & \textup{if $k\ge5$ is odd,}\\
2\cdot5^{(k-2)/2}, & \textup{if $k\ge4$ is even}
\end{cases}
;~
g(k-2)=
\begin{cases}
5^{(k-2)/2}, & \textup{if $k\ge6$ is even,}\\
2\cdot5^{(k-3)/2}, & \textup{if $k\ge5$ is odd,}
\end{cases}\]
where $g(k-2)=6$ if $k=4$. Set $t=R_2(H)-1$. For $H\in\mathscr{F}-\{H_{10}\}$, we have
\[
g(k-1)=
\begin{cases}
t\cdot 5^{(k-3)/2},& \textup{if $k\ge3$ is odd,}\\
4\cdot 5^{(k-2)/2},& \textup{if $k\ge2$ is even}
\end{cases}
;~
g(k-2)=
\begin{cases}
t\cdot 5^{(k-4)/2},& \textup{if $k\ge4$ is even,}\\
4\cdot 5^{(k-3)/2},& \textup{if $k\ge3$ is odd.}
\end{cases}
\] 
Thus by Theorem \ref{3}, $2g(k-1)\ge g(k-1)+g(k-2)+2\ge g(k-1)+k+3\ge2g(k-2)+4$ if $k\ge3$ for $H\in\mathscr{F}-\{H_{10}\}$ and $k\ge4$ for $H_{10}$. Furthermore,
\[\label{star}
g(k)+1>  
\begin{cases}
g(k-1)+k+1,& \textup{if $k\ge3$ and $H\in\mathscr{F}$,}\\
2g(k-1),& \textup{if $k\ge3$ for $H\in\mathscr{F}-\{H_{10}\}$ and $k\ge4$ for $H_{10}$,}\\
2g(k-1)+2,& \textup{if $k\ge3$ and $H\in\{H_5,H_6,H_{11}\}$,}\\
5g(k-2),& \text{if $k\ge3$ for $H\in\mathscr{F}-\{H_{10}\}$ and $k\ge5$ for $H_{10}$.}
\end{cases}
\tag{$\ast$}\] 

We next give the general lower bound of $GR_k((k-s)P_3,s\widehat{K}_4)$ for all $k\ge1$  and $s$ with $0\le s\le k$. Let 
\[w(k,s)=
\begin{cases}
2\cdot5^{s/2}, & \textup{if $s$ is even and $s<k$,}\\
2\cdot5^{s/2}-1, & \textup{if $s$ is even and $s=k$,}\\
4\cdot5^{(s-1)/2}, & \textup{if $s$ is odd,}
\end{cases}\]
for all $k\ge1$ and $s$ with $0\le s\le k$. We now construct the Gallai $k$-colored complete graph $G_k^s$ on $w(k,s)$ vertices recursively which contains neither monochromatic copy of $\widehat{K}_4$ in color $i\in[s]$ nor $P_3$ in color $j\in[k]\setminus[s]$.\medskip

If $s=0$, then by Theorem \ref{FD}, let $G^0_{k}$ be the Gallai $k$-colored complete graph on $w(k,0)=GR_k(P_3)-1=2$ vertices containing no $P_3$ in color $j\in[k]$. By Theorem \ref{3}, we see that $2(R_2(\widehat{K}_4)-1)<5(|\widehat{K}_4|-1)$. So if $s=k$, then $G_k^k$ can be obtained by Case \ref{CA} when $s$ is odd and by Case \ref{CASE} when $s$ is even.\medskip

Now it remains to consider the cases $1\le s<k$. By Theorem \ref{FD}, let $G^{'}$ be a Gallai $(k-s)$-colored complete graph on $2$ vertices with colors in $[k]\setminus [s]$ containing no monochromatic copy of $P_3$. Then we construct $G^{''}$ by joining two copies of $G^{'}$ such that all the edges between two copies are colored with color 1. Clearly, $G^{''}$ is a Gallai $(k-s+1)$-colored complete graph on 4 vertices which contains neither monochromatic copy of $\widehat{K}_4$ in color 1 nor $P_3$ in color $[k]\setminus [s]$. For $s=1$, let $G_k^1=G^{''}$. For all $s\ge2$, let $G_{k-2}^{s-2}$ be the construction on $w(k-2,s-2)$ vertices with colors in $[k]\setminus\{s-1,s\}$ such that $G_{k-2}^{s-2}=G^{'}$ if $s=2$. Let $K^{2}_5$ (a complete graph of order five) be colored with colors $s-1$ and $s$ without monochromatic triangle. Let $G_k^s$ be the construction obtained from the copies of $G_{k-2}^{s-2}$ by the same method as in Case \ref{C} under $K^{2}_5$. Since $G_{k-2}^{s-2}$ is a Gallai $(k-2)$-colored complete graph which contains neither monochromatic copy of $\widehat{K}_4$ in color $i\in[s-2]$ nor $P_3$ in color $j\in[k]\setminus [s]$, $G_k^s$ is a desired construction on $|G_k^s|=4\cdot5^{(s-1)/2}$ vertices for odd $s\ge1$ and $|G_k^s|=2\cdot5^{s/2}$ vertices for even $s\ge2$ without rainbow triangle.\medskip 

Therefore, $GR_k((k-s)P_3,s\widehat{K}_4)\ge w(k,s)+1$ for all $k\ge1$ and $s$ with $0\le s\le k$.\medskip

For all $k\ge3$ and $s$ with $1\le s\le k$, we have
\[w(k,s-1)=
\begin{cases}
4\cdot5^{(s-2)/2}, & \textup{if $s$ is even,}\\
2\cdot5^{(s-1)/2}, & \textup{if $s$ is odd,}
\end{cases}\]
and
\[w(k-1,s-1)=
\begin{cases}
4\cdot5^{(s-2)/2}, & \textup{if $s$ is even,}\\
2\cdot5^{(s-1)/2}-1, & \textup{if $s$ is odd and $s=k$,}\\
2\cdot5^{(s-1)/2}, & \textup{if $s$ is odd and $s<k$.}
\end{cases}\]
Thus $w(k,s-1)\ge w(k-1,s-1)\ge s+1$. For all $k\ge3$ and $s$ with $2\le s\le k$, we have 
\[w(k-1,s-2)=w(k,s-2)=
\begin{cases}
2\cdot5^{(s-2)/2}, & \textup{if $s$ is even,}\\
4\cdot5^{(s-3)/2}, & \textup{if $s$ is odd,}
\end{cases}\]
and 
\[w(k-2,s-2)=
\begin{cases}
2\cdot5^{(s-2)/2}, & \textup{if $s$ is even and $s<k$,}\\
2\cdot5^{(s-2)/2}-1, & \textup{if $s$ is even and $s=k$,}\\
4\cdot5^{(s-3)/2}, & \textup{if $s$ is odd.}
\end{cases}\]
Thus $w(k,s-2)=w(k-1,s-2)\ge w(k-2,s-2)\ge2$. Furthermore,
\[w(k,s)+1>
\begin{cases}
2w(k,s-1), & \textup{if $k\ge3$ and $1\le s\le k$,}\\
4w(k-1,s-2)+ w(k-2,s-2), & \textup{if $k\ge3$ and $2\le s\le k$.}
\end{cases}
\tag{$\star$}\]
\section{Proof of main results}
By the constructions in Section 2, it suffices to show that $GR_k(H)\le g(k)+1$ with $H\in\mathscr{F}$ and $GR_k((k-s)P_3,s\widehat{K}_4,)\le w(k,s)+1$ for all $k\ge1$ and $s$ with $0\le s\le k$. We proceed the proof by induction on $k$ for $GR_k(H)$ with $H\in\mathscr{F}$ and $k+s$ for $GR_k((k-s)P_3,s\widehat{K}_4)$. The case for $k=1$ is trivial. By Theorems \ref{3} and \ref{5}, $GR_2(H)=R_2(H)=g(2)+1$, $GR_2(\widehat{K}_4)=R_2(\widehat{K}_4)=w(2,2)+1$ and $GR_2(P_3,\widehat{K}_4)=R(P_3,\widehat{K}_4)=w(2,1)+1$. The case for $s=0$ is Theorem \ref{FD}. Therefore, we may assume that $k\ge3$ and $s$ with $1\le s\le k$. Suppose $GR_k(H)$ with $H\in\mathscr{F}$ holds for all $k^{'}<k$ and $GR_k((k-s)P_3,s\widehat{K}_4)$ holds for all $k^{'}+s^{'}<k+s$. Set $G=K_{g(k)+1}$ for $GR_k(H)$ with $H\in\mathscr{F}$ and $G=K_{w(k,s)+1}$ for $GR_k((k-s)P_3,s\widehat{K}_4)$. Let $c: E(G)\xrightarrow{} [k]$ be any Gallai $k$-coloring of $G$. Suppose $(G,c)$ contains no monochromatic copy of $H$ for $GR_k(H)$ with $H\in\mathscr{F}$ and $(G,c)$ contains neither monochromatic copy of $\widehat{K}_4$ in color $i\in[s]$ nor $P_3$ in color $j\in [k]\setminus[s]$ for $GR_k((k-s)P_3,s\widehat{K}_4)$. Choose $(G,c)$ with $k$ minimum.\medskip 

Let $A\subseteq V(G)$ and let $p,q\in[k]$ be two distinct colors. By induction for $GR_k(H)$ with $H\in\mathscr{F}$, if $(G[A],c)$ contains no edge in color $p$ ($resp$. $p$ and $q$), then $|A|\le g(k-1)$ ($resp$. $|A|\le g(k-2)$). Suppose $p,q\in[s]$ if $s\ge2$. By induction for $GR_k((k-s)P_3,s\widehat{K}_4)$, if $(G[A],c)$ contains no edge in color $p$ ($resp$. $p$ and $q$), then $|A|\le w(k-1,s-1)$ ($resp$. $|A|\le w(k-2,s-2)$). Moreover, if $(G[A],c)$ contains edges in color $p$ ($resp$. $p$ and $q$) but does not contain $P_3$ in color $p$ ($resp$. $p$ and $q$), then $|A|\le w(k,s-1)$ as $(G[A],c)$ contains neither monochromatic copy of $\widehat{K}_4$ in color $i\in[s]\setminus\{p\}$ nor $P_3$ in color $j\in([k]\setminus [s])\cup\{p\}$ ($resp$. $|A|\le w(k,s-2)$ as $(G[A],c)$ contains neither monochromatic copy of $\widehat{K}_4$ in color $i\in[s]\setminus\{p,q\}$ nor $P_3$ in color $j\in([k]\setminus [s])\cup\{p,q\}$). Furthermore, if $(G[A],c)$ contains edges in color $p$ but does not contain $P_3$ in color $p$ and edge in color $q$, then $|A|\le w(k-1,s-2)$. We will use these observations quite often in later proofs.\medskip

Let $u_1,u_2,\ldots,u_t\in V(G)$ be a maximum sequence of vertices chosen as follows: for each $j\in [t]$, all edges between $u_j$ and $V(G)\setminus\{u_1,u_2,\ldots,u_j\}$ are colored the same color under $c$. Let $U=\{u_1,u_2,\ldots,u_t\}$. Notice that $U$ is possibly empty. For each $u_j\in U$, let $c(u_j)$ be the unique color on the edges between $u_j$ and $V(G)\setminus\{u_1,u_2,\ldots,u_j\}$. 
\begin{claim}\label{c2}
	$c(u_i)\neq c(u_j)$ for all $i,j\in [t]$ with $i\neq j$.
\end{claim}
\begin{proof} 
Suppose that $c(u_i)=c(u_j)$ for some $i,j\in [t]$ with $i\neq j$. We may assume that $u_j$ is the first vertex in the sequence $u_1,\ldots,u_t$ such that $c(u_j)=c(u_i)$ for some $i\in [t]$ with $i<j$. We may further assume that the color $c(u_i)$ is red. Thus the edge $u_iu_j$ is colored with red under $c$. Let $W=V(G)\setminus\{u_1,u_2,\ldots,u_j\}$. Then all the edges between $\{u_i,u_j\}$ and $W$ are colored with red under $c$. We first consider the proof for $GR_k(H)$ with $H\in\mathscr{F}$. By the pigeonhole principle, $j\leq k+1$. Note that $(G[W],c)$ contains no red edge, otherwise we obtain a red $H\in\mathscr{F}$. By induction, $|W|\le g(k-1)$. By $(*)$, $|G|\le g(k-1)+k+1<g(k)+1$, which is a contradiction. We next consider the proof for $GR_k((k-s)P_3,s\widehat{K}_4)$. Note that the color on the edges between $u_j$ and $V(G)\setminus\{u_1,u_2,\ldots,u_j\}$ belongs to $[s]$, otherwise we obtain a monochromatic copy of $P_3$ with color in $[k]\setminus[s]$ as $|G|\ge w(k,s)+1\ge5$ for $k\ge3$ and $s$ with $1\le s\le k$. So red belongs to $[s]$. By the pigeonhole principle, $j\leq s+1$. By induction, $|W|\le w(k-1,s-1)$ as $(G[W],c)$ contains no red edge. Recall that $w(k,s-1)\ge w(k-1,s-1)\ge s+1$. By $(\star)$, $|G|\le w(k-1,s-1)+s+1\le 2w(k,s-1)<w(k,s)+1$, which is impossible.
\end{proof}

By Claim \ref{c2}, $|U|\le k$ for $GR_k(H)$ with $H\in\mathscr{F}$ and $|U|\le s$ for $GR_k((k-s)P_3,s\widehat{K}_4)$ as $c(u_i)$ belongs to $[s]$ for each $u_i\in U$. Consider a Gallai-partition of $G\setminus U$ with parts $V_1,V_2,\ldots,V_{\ell}$ such that $\ell\ge2$ is as small as possible. Assume that $|V_1|\ge|V_2|\ge \ldots \ge |V_{\ell}|$. Let $\mathcal{G}$ be the reduced graph of $G\setminus U$ with vertices $v_1,\ldots,v_{\ell}$. By Theorem \ref{Gallai}, we may further assume that the edges of $\mathcal{G}$ are colored with red or blue. It is obvious that any monochromatic copy of $H$ in $\mathcal{G}$ would yield a monochromatic copy of $H$ in $G\setminus U$ for $GR_k(H)$ with $H\in\mathscr{F}$ and any monochromatic copy of $\widehat{K}_4$ or $P_3$ in $\mathcal{G}$ would yield a monochromatic copy of $\widehat{K}_4$ or $P_3$ in $G\setminus U$ for $GR_k((k-s)P_3,s\widehat{K}_4)$. Let 
\begin{center}
$\mathcal{V}_r=\{V_i$ $|$ $V_i$ is red-adjacent to $V_1$ under $c$, $i\in\{2,\ldots,\ell\}\}$ and \\
$\mathcal{V}_b=\{V_i$ $|$ $V_i$ is blue-adjacent to $V_1$ under $c$, $i\in\{2,\ldots,\ell\}\}$.\ \ \ \ \
\end{center}
Let $R=\bigcup_{V_i\in\mathcal{V}_r}V_i$ and $B=\bigcup_{V_i\in\mathcal{V}_b}V_i$. Then $|G|=|V_1|+|R|+|B|+|U|$. Without loss of generality, we may assume that $|B|\leq |R|$. Obviously, $|R|\ge2$, otherwise the vertex in $R$ or $B$ can be added to $U$, contrary to the maximality of $t$ in $U$. As $(G[V_1\cup R],c)$ contains a red $P_3$, red belongs to $[s]$ for $GR_k((k-s)P_3,s\widehat{K}_4)$. It is worth noting that if $|V_1|\ge2$, then $(G[V_1\cup R],c)$ contains a red $C_4$ and thus no vertex in $U$ is red-adjacent to $V(G)\setminus U$ under $c$. It follows that $(G[U],c)$ contains no red edge. Furthermore, if $(G[V_1],c)$ doesn't contain red edge or red $P_3$, and neither does $(G[V_1\cup U],c)$ or $(G[V_1\cup B\cup U],c)$ when $|B|\le1$.
\begin{claim}\label{c5}
$|B|\ge1$ for $GR_k((k-s)P_3,s\widehat{K}_4)$ and thus $2\le s\le k$ and blue
 belongs to $[s]$.
\end{claim}
\begin{proof} 
Suppose $|B|=0$. Assume $\ell\ge3$. Since $R$ is red-adjacent to $V_1$ under $c$, we can actually find a Gallai-partition with only two parts, contrary to the fact that $\ell$ is as small as possible. Thus $\ell=2$. Since $|V_2|=|R|\ge2$, we have $|V_1|\ge2$. Thus no vertex in $U$ is red-adjacent to $V(G)\setminus U$ under $c$ and so $|U|\le s-1$. Clearly, there is no red $P_3$ within either $G[V_1]$ or $G[R]$ under $c$. Note that $|V_1|\ge3$, otherwise $|G|=|V_1|+|R|+|U|\le 4+s-1<w(k,s)+1$. If there exist red edges within $(G[V_1],c)$, then $(G[R],c)$ doesn't contain red edge, and neither does $(G[R\cup U],c)$. By induction, $|V_1|\le w(k,s-1)$ and $|R\cup U|\le w(k-1,s-1)$. Recall that $w(k-1,s-1)\le w(k,s-1)$. By $(\star)$, $|G|=|V_1|+|R\cup U|\le w(k,s-1)+w(k-1,s-1)\le2w(k,s-1)<w(k,s)+1$, which is impossible. So there is no red edge within $(G[V_1],c)$. By induction, $|V_1\cup U|\le w(k-1,s-1)$. By $(\star)$, $|G|=|V_1\cup U|+|V_2|\le 2w(k-1,s-1)\le2w(k,s-1)<w(k,s)+1$, contrary to the fact that $|G|=w(k,s)+1$. Thus $|B|\ge1$. Suppose $s=1$ for $k\ge3$. Then red is the only color in $[s]$. Note that $|V_1|=1$, otherwise there is a blue $P_3$ within $(G[V_1\cup B],c)$. Then $(G,c)$ only contains red and blue edges, contrary to our assumption that $k\ge3$. Thus $2\le s\le k$. Suppose blue is a color in $[k]\setminus[s]$. By Theorem \ref{5}, $R(P_3,\widehat{K}_4)=5$ and so $\ell\le4$. Thus $|V_1|\ge2$ as $|G\setminus U|\ge w(3,2)-1=9$. But there will be a blue $P_3$ within $(G[V_1\cup B],c)$, which is a contradiction.
\end{proof}

\begin{claim}\label{c1}
Let $C,D$ be two disjoint sets of $V(G)$ such that $|C|\ge3$, $|D|=2$ and $C$ is red or blue-adjacent to $D$ under $c$. For $GR_k(H)$ with $H\in\mathscr{F}$, there is no red or blue edge within $(G[C],c)$ if $H\in\mathscr{F}-\{H_6\}$ $($$|C|\ge4$ for $H=H_2$$)$ or $(G[D],c)$ if $H\in\{H_2,H_6\}$.
\end{claim}
\begin{proof} 
Suppose not. Then there will be a red or blue $H\in\mathscr{F}$ in $(G[C\cup D],c)$, contrary to our assumption that $(G,c)$ contains no monochromatic copy of $H\in\mathscr{F}$.
\end{proof}

\begin{claim}\label{c3}
$|V_1|\ge3$ for $GR_k(H)$ with $H\in\mathscr{F}$ and $|V_1|\ge2$ for $GR_k((k-s)P_3,s\widehat{K}_4)$.
\end{claim}
\begin{proof} 
We first consider the proof for $GR_k(H)$ with $H\in\mathscr{F}$. By Theorem \ref{3}, $\ell\le6$ if $H=H_{10}$ and $\ell\le9$ if $H\in\mathscr{F}-\{H_{10}\}$. Then $|V_1|\ge2$ as $|G\setminus U|\ge g(3)-2=8$ if $H=H_{10}$ and $|G\setminus U|\ge g(3)-2=18$ if $H\in\mathscr{F}-\{H_{10}\}$. Since $|R|\ge2$, no vertex in $U$ is red-adjacent to $V(G)\setminus U$ under $c$ and so $|U|\le k-1$. Thus $|V_1|\ge3$ if $H\in\mathscr{F}-\{H_{10}\}$ as $|G\setminus U|\ge g(3)-1=19$. Note that if $H=H_{10}$ and $k\ge4$, then $|V_1|\ge3$ as $|G\setminus U|\ge g(4)-2=23$. So we next assume that $H=H_{10}$ and $k=3$. Suppose $|V_1|=2$. Then $|R|\ge3$ as $|G\setminus U|\ge8$ and $|R|\ge|B|$. By Claim \ref{c1}, $(G[R],c)$ doesn't contain red edge and neither does $(G[R\cup U],c)$. By induction, $|R\cup U|\le g(2)=6$. It follows that $|B|\ge3$ as $|G|=g(3)+1=11$. Thus no vertex in $U$ is blue-adjacent to $V(G)\setminus U$ under $c$ and so $|U|\le1$. By Claim \ref{c1} again, $(G[B],c)$ contains no blue edge. Now, we see that the color on the edges between any pairs of $\{V_2,\ldots,V_{\ell}\}$ in $(G[R],c)$ and $(G[B],c)$ is blue and red, respectively.  Note that there is no blue $K_3$ in $(G[R],c)$ as $(G[V_1\cup B],c)$ contains blue edges. Then there are at most two parts of $\{V_2,\ldots,V_{\ell}\}$ in $(G[R],c)$, which means that $|B|\le |R|\le4$. Thus $|U|=1$ as $|G|=11$. Since $|V_1|=2$, $(G[V_1],c)$ only contains one color. Such color is different from red and blue, otherwise there is a red $K_3$ in $(G[V_1\cup R],c)$ or a blue $K_3$ in $(G[V_1\cup B],c)$, which together with a red edge in $(G[B],c)$ or a blue edge in $(G[R],c)$ yields a monochromatic copy of $H_{10}$. Let green denote such color. Then $(G[V_1\cup U],c)$ contains a green $K_3$, which forbids green edge in $(G[R\cup B],c)$. By induction, $|R\cup B|\le g(2)=6$. Then $|G|=|V_1\cup U|+|R\cup B|\leq 9<g(3)+1$, contrary to the fact that $|G|=g(k)+1$.\medskip

We next consider the proof for $GR_k((k-s)P_3,s\widehat{K}_4)$. By Claim \ref{c5}, we only need to assume that $3\le s\le k$ for $k\ge3$ as one color other than red and blue must occur in $(G\setminus U,c)$ when $s=2$, which forces $|V_1|\ge2$. By Theorem \ref{3}, $R_2(\widehat{K}_4)=10$. By Claim \ref{c5}, $\ell\le9$ and so $|V_1|\ge2$ as $|G\setminus U|\ge w(k,3)-2\ge18$ for all $k\ge3$ and $3\le s\le k$.
\end{proof}

For the convenience of the proofs, we now show that the desired result holds for $H=H_{10} $ and $k=3$. By Claims \ref{c1} and \ref{c3}, $(G[V_1],c)$ contains no red edge. Suppose $|B|\le1$. Then there is no red edge in $(G[V_1\cup B\cup U],c)$. By induction, $|V_1\cup B\cup U|\le g(2)=6$ and so $|R|\ge3$ as $|G|=11$. By Claims \ref{c1} and \ref{c3} again, there is no red edge in $(G[R],c)$. By induction, $|R|\le g(2)=6$. Since $R_2(K_3)=6$, we have $|V_1\cup B\cup U|\le5$ and $|R|\le5$. Then $|G|=|V_1\cup B\cup U|+|R|\le 10<g(3)+1$, which is impossible. Suppose $|B|\ge2$. By Claims \ref{c1} and \ref{c3}, there is no blue edge in $(G[V_1],c)$ and so $(G[V_1\cup U],c)$ contains only one color which is different from red and blue, say green. It follows that $3\le|V_1\cup U|\le4$ and so $(G[V_1\cup U],c)$ contains a green $K_3$, which forbids green edge in $(G[R\cup B],c)$. By induction, $|R\cup B|\le g(2)=6$. Then $|G|=|V_1\cup U|+|R\cup B|\le 10<g(3)+1$, which is a contradiction. Thus we may assume that $k\ge4$ if $H=H_{10}$ in later proofs.
\begin{claim}\label{c4}
$|B|\ge3$ for $GR_k(H)$ with $H\in\mathscr{F}$ and $|B|\ge2$ for $GR_k((k-s)P_3,s\widehat{K}_4)$.
\end{claim}
\begin{proof}
We first suppose $|B|\le1$ for both cases. Recall that $|R|\ge2$. By Claim \ref{c3}, there is no red $P_3$ within either $(G[R],c)$ or $(G[V_1\cup B\cup U],c)$ for $GR_k((k-s)P_3,s\widehat{K}_4)$. By induction, $|R|\le w(k,s-1)$ and $|V_1\cup B\cup U|\le w(k,s-1)$. By $(\star)$, $|G|=|V_1\cup B\cup U|+|R|\le 2w(k,s-1)<w(k,s)+1$, which is a contradiction. We next consider the case for $GR_k(H)$ with $H\in\mathscr{F}$. Assume that $|R|=2$. Let $H\in\mathscr{F}-\{H_6\}$. Note that $|V_1|\ge4$ as $|G|\ge21$. By Claim \ref{c1}, there is no red edge within $(G[V_1\cup B\cup U],c)$, By induction, $|V_1\cup B\cup U|\le g(k-1)$. By $(*)$, $|G|=|V_1\cup B\cup U|+|R|\leq g(k-1)+2\le g(k-1)+k+1<g(k)+1$, which is impossible. Let $H=H_6$. Clearly, there is no red $P_3$ within $(G[V_1\cup B\cup U],c)$. Then all the red edges in $(G[V_1\cup B\cup U],c)$ induce a matching, say $M$. Let $X_1,X_2$ be two disjoint sets of $V_1\cup B\cup U$ such that $X_1\cup X_2=V_1\cup B\cup U$ and $M$ is part of the edges between $X_1$ and $X_2$ under $c$. Then there is no red edge within either $(G[X_1],c)$ or $(G[X_2],c)$. By induction, $|V_1\cup B\cup U|\le 2g(k-1)$. By $(*)$, $|G|=|V_1\cup B\cup U|+|R|\leq 2g(k-1)+2<g(k)+1$, which is a contradiction. Assume that $|R|\ge3$. By Claims \ref{c1} and \ref{c3}, there is no red edge within either $(G[V_1\cup B\cup U],c)$ or $(G[R],c)$. By induction, $|V_1\cup B\cup U|\le g(k-1)$ and $|R|\le g(k-1)$. By $(*)$, $|G|=|V_1\cup B\cup U|+|R|\leq 2g(k-1)<g(k)+1$, contrary to the fact that $|G|=g(k)+1$.\medskip 
	
We next suppose $|B|=2$ for $GR_k(H)$ with $H\in\mathscr{F}$. Then no vertex in $U$ is red or blue-adjacent to $V(G)\setminus U$ under $c$ and so $|U|\le k-2$. Recall that $2g(k-1)\ge g(k-1)+g(k-2)+2\ge g(k-1)+k+3\ge 2g(k-2)+4$. Assume that $|R|=2$. Let $H\in\mathscr{F}-\{H_6\}$. By Claims \ref{c1} and \ref{c3}, there is neither red nor blue edge within $(G[V_1\cup U],c)$. By induction, $|V_1\cup U|\le g(k-2)$.  By $(*)$, $|G|=|V_1\cup U|+|R|+|B|\leq g(k-2)+4\le 2g(k-1)<g(k)+1$, which is a contradiction. Let $H=H_6$. Obviously, there is neither red $P_3$ nor blue $P_3$ within $(G[V_1\cup U],c)$, Similar to the arguments for $|B|\le1$, Since $(G,c)$ has no rainbow triangle, all the red and blue edges in $(G[V_1\cup U],c)$ induce a matching. Then there is neither red nor blue edge within either $(G[X_1],c)$ or $(G[X_2],c)$. By induction, $|V_1\cup U|\le 2g(k-2)$. By $(*)$, $|G|=|V_1\cup U|+|R|+|B|\leq 2g(k-2)+4\le 2g(k-1)<g(k)+1$, which is impossible. Assume that $|R|\ge3$. By Claims \ref{c1} and \ref{c3}, there is no red edge within either $(G[R],c)$ or $(G[V_1\cup U],c)$. By induction, $|R|\le g(k-1)$ and $|V_1\cup U|\le g(k-1)$. Let $H\in\{H_5,H_6,H_{11}\}$. By $(*)$, $|G|=|V_1\cup U|+|R|+|B|\leq 2g(k-1)+2<g(k)+1$, which is a contradiction. Let $H\in\mathscr{F}-\{H_5,H_6,H_{11}\}$.  Note that $|V_1|\ge4$, otherwise $|G|=|V_1|+|R|+|B|+|U|\le g(k-1)+k+3\le2g(k-1)<g(k)+1$. By Claim \ref{c1}, there is no blue edge within $(G[V_1\cup U],c)$. By induction, $|V_1\cup U|\le g(k-2)$. By $(*)$, $|G|=|V_1\cup U|+|R|+|B|\leq g(k-2)+g(k-1)+2\le 2g(k-1)<g(k)+1$, which yields a contradiction.
\end{proof}

By Claims \ref{c1}-\ref{c4}, there is no red edge ($resp$. $P_3$) within $(G[R],c)$ and no blue edge ($resp$. $P_3$) within $(G[B],c)$ for $GR_k(H)$ with $H\in\mathscr{F}$ ($resp$. $GR_k((k-s)P_3,s\widehat{K}_4)$). Define $Y_1=\{V_i:|V_i|=1, i\in\{2,\ldots,\ell\}\}$ and $Y_2=\{V_i: |V_i|\ge2, i\in\{2,\ldots,\ell\}\}$. Then $|Y_1\cup Y_2|=|R\cup B|$. Let $|R\cap Y_t|$ and $|B\cap Y_t|$ be the number of the common parts in $\{V_2,\ldots,V_{\ell}\}$, where $t=1,2$. We see that $|R\cap Y_2|\le2$ and $|B\cap Y_2|\le2$, otherwise there is a blue $H\in \mathscr{F}\cup \widehat{K}_4$ in $(G[R],c)$ or a red $H\in \mathscr{F}\cup \widehat{K}_4$ in $(G[B],c)$. Furthermore, we have the following three facts:
\begin{enumerate}
\item If $|R\cap Y_2|=2$ or $|B\cap Y_2|=2$, then $|R\cap Y_1|=0$ or $|B\cap Y_1|=0$.
\item If $|R\cap Y_2|=1$ or $|B\cap Y_2|=1$, then $|R\cap Y_1|\le2$ or $|B\cap Y_1|\le2$.	
\item If $|R\cap Y_2|=0$ or $|B\cap Y_2|=0$, then $|R|\le4$ or $|B|\le4$.	
\end{enumerate}
\begin{proof}
We only consider the proof for $R$. The proof for $B$ is similar. Suppose $|R\cap Y_1|\ge1$ for fact 1. Then $Y_1$ is blue-adjacent to $Y_2$ and so there is a blue $H\in \mathscr{F}\cup \widehat{K}_4$ in $(G[R],c)$, which is impossible. Suppose $|R\cap Y_1|\ge3$ for fact 2. Let $V_i,V_j,V_k\in Y_1$ and let $V_s\in Y_2$. Then $V_i,V_j$ and $V_k$ are blue-adjacent to $V_s$ and there is at most one red edge in $(G[ V_i\cup V_j\cup V_k],c)$. Thus there is a blue $H\in \mathscr{F}\cup \widehat{K}_4$ in $(G[V_s\cup V_i\cup V_j\cup V_k],c)$, which is impossible. Suppose $|R|\ge5$ for fact 3. Then $(G[R],c)$ must contain red and blue edges, contrary to the fact that there is no red edge within $(G[R],c)$ for $GR_k(H)$ with $H\in\mathscr{F}$. For $GR_k((k-s)P_3,s\widehat{K}_4)$, since $(G[R],c)$ contains no red $P_3$ and $R(P_3,\widehat{K}_4)=5$, there is a blue $\widehat{K}_4$ in $(G[R],c)$, which is a contradiction.
\end{proof}

By facts 1-3, it is easily seen that $|R|\le 2|V_1|$ as $4\le |V_1|+2\le 2|V_1|$. Then $|B|\le |R|\le2|V_1|$. By Claims \ref{c3} and \ref{c4}, no vertex in $U$ is red or blue-adjacent to $V(G)\setminus U$ under $c$ and thus $(G[U],c)$ contains neither red nor blue edge.\medskip

We first consider the proof for $GR_k(H)$ with $H\in\mathscr{F}$. By Claims \ref{c1}-\ref{c4}, $(G[V_1\cup U],c)$ contains neither red nor blue edge. By induction, $|V_1\cup U|\le g(k-2)$. Assume that $k\ge5$ if $H=H_{10}$ and $k\ge3$ if $H\in\mathscr{F}-\{H_{10}\}$. As $|R|\le 2|V_1|$, $|R|+|B|\leq 4g(k-2)$. By $(*)$, $|G|=|V_1\cup U|+|R|+|B|\leq 5g(k-2)<g(k)+1$, which is a contradiction. Assume that $H=H_{10}$ and $k=4$. Then $|V_1\cup U|\le 5$, otherwise $(G[V_1\cup U],c)$ contains a monochromatic $K_3$, say green, as $R_2(K_3)=6$, which forbids green edge in $(G[R\cup B],c)$. By induction, $|R\cup B|\le g(3)=10$. Then $|G|=|V_1\cup U|+|R\cup B|\leq 16<g(4)+1$, which is impossible. Now, $|G|=|V_1\cup U|+|R|+|B|\leq 25<g(4)+1$, contrary to the fact that $|G|=g(k)+1$. We next turn to the proof for $GR_k((k-s)P_3,s\widehat{K}_4)$. We know that $|R|\ge2$. By Claim \ref{c4}, $(G[V_1\cup U],c)$ contains neither red $P_3$ nor blue $P_3$. By induction, $|V_1\cup U|\le w(k,s-2)$. Recall that $w(k,s-2)=w(k-1,s-2)\ge w(k-2,s-2)\ge2$. Note that $|R\cap Y_2|\le2$ and $|B\cap Y_2|\le2$. Suppose $|R\cap Y_2|\le1$. By facts 2 and 3, $|B|\le|R|\le |V_1|+2$. By $(\star)$, $|G|=|V_1\cup U|+|R|+|B|\le 3w(k,s-2)+4\le 4w(k-1,s-2)+w(k-2,s-2)<w(k,s)+1$, which is impossible. Thus $|R\cap Y_2|=2$. By facts 2 and 3 again, $|B|\le|V_1|+2$ if $|B\cap Y_2|\le1$. By $(\star)$, $|G|=|V_1\cup U|+|R|+|B|\le 4w(k,s-2)+2\le 4w(k-1,s-2)+w(k-2,s-2)<w(k,s)+1$, which is a contradiction. Thus $|B\cap Y_2|=2$. By fact 1, $|R\cap Y_1|=0$ and $|B\cap Y_1|=0$. Then $\ell=5$. Since $R_2(\widehat{K}_4)=10$, there is no $K_{10}$ in $G\setminus U$ whose edges are colored with red and blue under $c$, which implies that at least one part of $\{V_1,\ldots,V_5\}$ contains neither red nor blue edge. Note that $|V_1\cup U|\le w(k,s-2)=w(k-1,s-2)$ and $|V_1|\ge|V_2|\ge\ldots\ge |V_5|$. By $(\star)$, $|G|=|V_1\cup U|+|R|+|B|\le 4w(k-1,s-2)+w(k-2,s-2)<w(k,s)+1$, contrary to the fact that $|G|=w(k,s)+1$.\medskip

This completes the proofs of Theorems \ref{thm} and \ref{thmmm}.\qed

\section*{Acknowledgements}

The authors are grateful to the anonymous referee for helpful comments and suggestions which improved the presentation of this paper.

\end{document}